\documentclass{amsart}

\usepackage[T1]{fontenc}
\usepackage[utf8]{inputenc}
\usepackage{amssymb,mathtools}
\usepackage{microtype}
\usepackage[hidelinks]{hyperref}
\usepackage{tikz}
\usetikzlibrary{fadings,patterns.meta}

\theoremstyle{plain}
\newtheorem{theorem}{Theorem}[section]
\newtheorem{lemma}[theorem]{Lemma}
\newtheorem{proposition}[theorem]{Proposition}
\newtheorem{corollary}[theorem]{Corollary}
\newtheorem{conjecture}[theorem]{Conjecture}
\theoremstyle{remark}
\newtheorem{remark}[theorem]{Remark}
\numberwithin{equation}{section}

\title[Gautschi's Conjecture]
{A Proof of Gautschi's Conjecture on Subrange Jacobi Polynomials}

\author{V. Botta}

\address{Department of Mathematics and Computer Science,
S\~ao Paulo State University (UNESP), Presidente Prudente, SP, Brazil}

\email{vanessa.botta@unesp.br}

\author{K. Castillo}

\address{CMUC, Department of Mathematics, University of Coimbra,
3000-143 Coimbra, Portugal}

\email{math@keniercastillo.com}

\author{L. Tertuliano da Silva}

\address{Department of Mathematics and Computer Science,
S\~ao Paulo State University (UNESP), Presidente Prudente, SP, Brazil}

\email{lucas.tertuliano@unesp.br}

\date{}



\begin{document}

\begin{abstract}
Let \(\pi_n\) be the monic polynomial of degree \(n\) orthogonal on
\([-c,c]\), \(0<c\leq1\), with respect to the Jacobi weight
\((1-x)^\alpha(1+x)^\beta\), where \(-1<\alpha<\beta\).
Gautschi conjectured that
\[
\left[
\frac{\pi_n(-c)}{\pi_n(c)}
\right]^2
\left(\frac{1-c}{1+c}\right)^{\beta-\alpha}
<1.
\]
By his variation formula, this inequality is sufficient for every
positive zero of \(\pi_n\) to move to the right as \(c\) increases.
For \(0<c<1\), a first-crossing argument proves the conjecture
throughout \(0<\alpha<\beta\). Together with the earlier regions
recorded by Gautschi and established by Milovanovi\'c, this settles
\(\beta\geq0\). In the negative wedge, writing
\(\alpha=-r-\lambda\) and \(\beta=-r+\lambda\), an ensemble Ward
bound yields the region \(c^2\leq3/(3+r)\). A strengthening of the
same crossing lemma, using an orthogonal expansion and Markov's
theorem, removes this restriction.
Consequently, the conjecture holds for every \(n\geq1\),
\(-1<\alpha<\beta\), and \(0<c\leq1\). We also give a direct degree-one proof and an explicit asymptotic
limit. The case
\(c=1\) is immediate.
\end{abstract}

\maketitle

\section{Introduction}

Throughout the paper we work in the real regime
\[
-1<\alpha<\beta,
\quad
0<c<1,
\]
where the Jacobi weight restricted to \([-c,c]\) is strictly positive
and bounded. This is the underlying admissible parameter range.

Subrange systems restrict the Jacobi weight to $[-c,c]$.
Subrange Chebyshev polynomials were used by Da Fies and Vianello in
subperiodic trigonometric quadrature \cite{DaFiesVianello2012};
Gautschi also developed subrange systems for the Jacobi weight
\cite{Gautschi2012}; see also the subsequent erratum and correction
\cite{Gautschi2017,Gautschi2019}. Although the truncation is elementary
at the level of the measure, the resulting orthogonal polynomials are generally no longer Jacobi polynomials and acquire a non-trivial
dependence on the geometric parameter $c$. The zeros are the nodes of the Gaussian quadrature rule for the
truncated Jacobi measure. When monotonicity with respect to \(c\) is
available, it orders the positive nodes associated with two interval
lengths and supplies a priori brackets for numerical continuation.
Uniformity in the degree is therefore substantially stronger than
verification at any prescribed collection of degrees; see
\cite{Gautschi2004} for the computational theory of orthogonal
polynomials and Gaussian quadrature.

Gautschi later investigated the motion of the zeros under variation
of $c$. In the ultraspherical case
$\alpha=\beta$, he proved that every positive zero moves strictly to
the right as the interval expands
\cite[Theorem~1]{Gautschi2018}. For the symmetric subrange with an
asymmetric Jacobi weight, $\alpha<\beta$, his variation formula reduces
the corresponding question to a comparison of the two weighted
endpoint values. This led to the following conjecture.

Let
\[
w(x)=(1-x)^\alpha(1+x)^\beta,
\quad -1<\alpha<\beta,
\]
and let $\pi_n=\pi_n^{(\alpha,\beta)}(\,\cdot\,;c)$ be the monic
polynomial of degree $n$ satisfying
\[
\int_{-c}^{c}\pi_n(x)x^j w(x)\,dx=0,
\quad 0\leq j<n.
\]
All the zeros of $\pi_n$ are simple and belong to $(-c,c)$.

\begin{conjecture}[Gautschi, 2018]\label{conj:gautschi}
For every $n\geq1$, $-1<\alpha<\beta$, and $0<c\leq1$,
\begin{equation}
\left[
\frac{\pi_n(-c)}{\pi_n(c)}
\right]^2
\left(\frac{1-c}{1+c}\right)^{\beta-\alpha}
<1.
\label{eq:gautschi}
\end{equation}
\end{conjecture}

For $c=1$ the conjecture is immediate, since $\beta-\alpha>0$ and
$\pi_n(1)\neq0$. Hence only $0<c<1$ is non-trivial, and this
restriction is understood throughout the remainder.

Its relevance to zero motion follows from Gautschi's variation
formula. If $x_{j,n}(c)>0$ is a zero of $\pi_n$, that formula
\cite[Equation~(2.5); cf.~Theorem~2]{Gautschi2018} gives
\begin{equation}
\begin{split}
&w(c)\pi_n(c)^2
\left\{
\frac{1}{c-x_{j,n}}
-
\left[
\frac{\pi_n(-c)}{\pi_n(c)}
\right]^2
\frac{w(-c)}{w(c)}
\frac{1}{c+x_{j,n}}
\right\}
\\[7pt]
&\hspace{35mm}
=
A_{j,n}(c)\pi_n'(x_{j,n})^2
\frac{dx_{j,n}}{dc},
\end{split}
\end{equation}
where $A_{j,n}(c)>0$ is the corresponding Gaussian quadrature
weight, or Christoffel number.
If $x_{j,n}>0$, replacing the factor supplied by
\eqref{eq:gautschi} by $1$ decreases the expression in braces to
$2x_{j,n}/(c^2-x_{j,n}^2)>0$. Hence $dx_{j,n}/dc>0$, and every
zero has positive derivative at every value of $c$ for which it is
positive. This endpoint inequality is therefore a sufficient, but not
equivalent, condition for the sign of that derivative.

Gautschi supported the conjecture with extensive MATLAB experiments
\cite[Appendix~B]{Gautschi2018}. The plots presented there are
consistent with the stronger behaviour that the endpoint expression
decreases from its limiting value $1$ as $c\downarrow0$. On a
logarithmic scale the quantity to be tested is
\[
\Phi_n(c)
=
\log\left\lvert\frac{\pi_n(-c)}{\pi_n(c)}\right\rvert
-(\beta-\alpha)\operatorname{arctanh}c,
\]
and the conjecture is $\Phi_n(c)<0$. This form is numerically stable,
but direct differentiation introduces derivatives of the
$c$-dependent orthogonal polynomial and does not produce a closed
sign estimate. The structural approach below replaces the quotient
by an endpoint difference and, later, by equivalent expressions in
the recurrence coefficients and in a reciprocal moment of an induced
probability measure. No
numerical evidence is used in the proofs.

We emphasise an important point of provenance. Under the standing
assumption $\alpha<\beta$, Gautschi recorded the endpoint inequality
for $\alpha\leq0\leq\beta$, citing an unpublished communication by
Milovanovi\'c to Gautschi
\cite[Remark, item~2, p.~763]{Gautschi2018}. Milovanovi\'c
subsequently supplied a published proof of the more general criterion
\[
\beta-\alpha\geq c\lvert\alpha+\beta\rvert
\]
\cite[Equation~(4) and Theorem~2.1]{Milovanovic2022}. Beyond the
sector recorded by Gautschi, this criterion adds one wedge in each
same-sign sector.

The boundary identity \eqref{eq:boundary-identity}, established in
Section~\ref{sec:boundary}, shows
why the two earlier sign arguments are especially direct. In
Gautschi's recorded sector,
\[
\frac{w'(x)}{w(x)}
=
-\frac{\alpha}{1-x}+\frac{\beta}{1+x}>0,
\quad -c<x<c.
\]
Each summand is non-negative and at least one is strictly positive.
The conjecture follows by a single integration because orthogonality
eliminates the term containing $\pi_n'$. Milovanovi\'c's enlargement
requires only the affine rewrite
\[
\frac{w'(x)}{w(x)}
=
\frac{\beta-\alpha-(\alpha+\beta)x}{1-x^2},
\]
whose numerator is non-negative throughout $[-c,c]$ precisely when
\begin{equation}
\beta-\alpha\geq c\lvert\alpha+\beta\rvert.
\label{eq:known-criterion}
\end{equation}
Thus both earlier results reduce to a pointwise sign test for $w'$.
They are automatically uniform in $n$ and require no estimate for the
orthogonal polynomial itself. This elementary mechanism also explains
exactly why the sign argument stops when the affine numerator changes
sign inside the interval.

For the remainder, we keep $0<c<1$ and define
\begin{equation}
\rho_c=\frac{1-c}{1+c}.
\label{eq:rho}
\end{equation}
Our first new result treats the part of the positive sector in which
the affine numerator changes sign:
\[
0<\alpha<\beta<\rho_c^{-1}\alpha,
\]
where pointwise monotonicity of the weight cannot settle the
conjecture. We regard the endpoint difference as a real-analytic
function of $c$ and rule out its first possible zero by coupling an
identity obtained by integration with a root-motion formula. This proves the
conjecture, for every degree, throughout $0<\alpha<\beta$ and hence
settles the entire admissible range $\beta\geq0$. The first-crossing mechanism does not require an explicit
representation of the subrange polynomials: it converts the global
endpoint inequality into a local transversality statement controlled
by the sum of the zeros.

The same crossing mechanism can be coupled with positive association
in the orthogonal polynomial ensemble. The Ward estimate in
Section~\ref{sec:ensemble} yields a degree-uniform slice of the
negative wedge. Combined with the regions already established, this
covers every admissible parameter pair when $0<c\leq\sqrt3/2$.
Remark~\ref{rem:crossing-addendum} extends
the coefficient calculation in the proof of Lemma~\ref{lem:integration-identity} and applies
Markov's theorem on the monotonicity of zeros. This removes the
root-sum hypothesis from the crossing lemma and extends the proof
to the remaining region $\mathcal O_c$. We retain the direct degree-one proof
and identify the large-degree limit of the normalised endpoint
difference in Section~\ref{sec:asymptotic}.

For a compact statement of the results, retain the notation
\eqref{eq:rho}.
All parameter sets below are understood to be subsets of the
admissible parameter domain $-1<\alpha<\beta$. The sector explicitly
recorded by Gautschi is
\begin{equation}
\mathcal G=
\{\alpha\leq0\leq\beta\}.
\label{eq:gautschi-region}
\end{equation}
The same-sign parts added by Milovanovi\'c's criterion
\eqref{eq:known-criterion} are
\begin{equation}
\begin{split}
\mathcal M_c^{-}
:={}&
\{\alpha<\beta<0,\ \beta\geq\rho_c\alpha\},
\\[7pt]
\mathcal M_c^{+}
:={}&
\{0<\alpha<\beta,\ \beta\geq\rho_c^{-1}\alpha\},
\\[7pt]
\mathcal M_c
:={}&\mathcal M_c^{-}\cup\mathcal M_c^{+}.
\end{split}
\label{eq:milovanovic-increment}
\end{equation}
Thus the full region supplied by Milovanovi\'c's criterion is the
union $\mathcal G\cup\mathcal M_c$. The positive-sector
portion established here and not already contained in that union is
\begin{equation}
\mathcal N_c
:=
\{0<\alpha<\beta<\rho_c^{-1}\alpha\}.
\label{eq:new-region}
\end{equation}
Inside the negative wedge, the range supplied by the Ward estimate is
\begin{equation}
\mathcal P_c
:=
\left\{
-1<\alpha<\beta<\rho_c\alpha<0,\quad
\alpha+\beta\geq-\frac{6(1-c^2)}{c^2}
\right\},
\label{eq:ensemble-region}
\end{equation}
whereas its complement, covered by the strengthened crossing lemma, is
\begin{equation}
\mathcal O_c
:=
\left\{
-1<\alpha<\beta<\rho_c\alpha<0,\quad
\alpha+\beta<-\frac{6(1-c^2)}{c^2}
\right\}.
\label{eq:open-region}
\end{equation}
The five sets
$\mathcal G$, $\mathcal M_c$, $\mathcal N_c$, $\mathcal P_c$, and
$\mathcal O_c$ form a disjoint partition of the admissible parameter
domain. We denote this disjoint decomposition by
\[
\mathcal G\,\sqcup\,
\mathcal M_c\,\sqcup\,
\mathcal N_c\,\sqcup\,
\mathcal P_c\,\sqcup\,
\mathcal O_c.
\]
This partition distinguishes the earlier regions, the range of the Ward
estimate, and the complementary region covered by the addendum.

\begin{theorem}\label{thm:global}
Conjecture~\ref{conj:gautschi} holds for every $n\geq1$,
$-1<\alpha<\beta$, and $0<c\leq1$.
\end{theorem}

For $0<c<1$, this covers all five regions in the partition above.
The proof on $\mathcal G\cup\mathcal M_c$ uses the earlier criterion,
and Proposition~\ref{prop:positive-sector} settles $\mathcal N_c$.
On $\mathcal P_c$, the Ward estimate supplies the root-sum condition
required by the crossing lemma. The extension to $\mathcal O_c$
requires the refined crossing lemma in
Remark~\ref{rem:crossing-addendum}, which removes that condition.
Lemma~\ref{lem:degree-one} supplies a direct base case for the
induction. Proposition~\ref{prop:asymptotic-limit} identifies the
large-degree limit of the normalised endpoint difference for each
fixed admissible triple with $0<c<1$.

The geometry and provenance of the degree-uniform regions are
summarised in Figure~\ref{fig:parameter-range}.

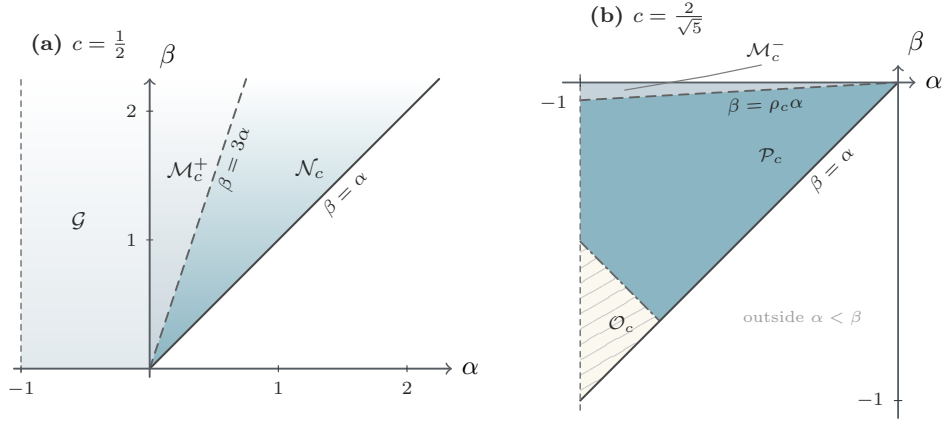
\begin{figure}[!htbp]
\centering
\definecolor{regiong}{RGB}{226,233,237}
\definecolor{regionm}{RGB}{198,211,218}
\definecolor{regionnew}{RGB}{139,181,194}
\definecolor{regionopen}{RGB}{251,248,240}
\definecolor{axisgray}{RGB}{86,93,101}
\definecolor{textgray}{RGB}{40,40,40}

\noindent
\begin{minipage}{0.52\textwidth}
\centering
\begin{tikzpicture}[
  line cap=round,
  line join=round
]
\tikzset{
  axis/.style={
    draw=axisgray,
    line width=0.72pt,
    ->
  },
  boundary/.style={
    draw=black!72,
    line width=0.78pt
  },
  criterion/.style={
    draw=black!62,
    line width=0.72pt,
    dash pattern=on 3.6pt off 2.6pt
  },
  excluded/.style={
    draw=black!52,
    line width=0.58pt,
    dash pattern=on 2.0pt off 2.0pt
  },
  linelabel/.style={
    font=\scriptsize,
    text=textgray,
    inner sep=0pt
  },
  regionlabel/.style={
    font=\footnotesize,
    text=textgray,
    inner sep=0pt,
    align=center
  },
  regiontext/.style={
    font=\scriptsize,
    text=textgray,
    inner sep=0pt,
    align=center
  }
}

\begin{scope}[x=1.70cm,y=1.70cm]
\begin{scope}
\clip (-1.05,-0.08) rectangle (2.35,2.25);

\begin{scope}
\clip
  (-1,0) --
  (0,0) --
  (0,2.25) --
  (-1,2.25) --
  cycle;
\shade[
  bottom color=regiong,
  top color=white
]
  (-1,0) rectangle (0,2.25);
\end{scope}

\begin{scope}
\clip
  (0,0) --
  (0.75,2.25) --
  (0,2.25) --
  cycle;
\shade[
  bottom color=regionm,
  top color=white
]
  (0,0) rectangle (0.75,2.25);
\end{scope}

\begin{scope}
\clip
  (0,0) --
  (0.75,2.25) --
  (2.25,2.25) --
  cycle;
\shade[
  bottom color=regionnew,
  top color=white
]
  (0,0) rectangle (2.25,2.25);
\end{scope}
\end{scope}

\draw[criterion]
  (0,0) -- (0.75,2.25)
  node[
    pos=0.74,
    sloped,
    below=2.0pt,
    linelabel
  ]
  {$\beta=3\alpha$};

\draw[boundary]
  (0,0) -- (2.25,2.25)
  node[
    pos=0.64,
    sloped,
    below=2.0pt,
    linelabel
  ]
  {$\beta=\alpha$};

\draw[excluded]
  (-1,-0.03) -- (-1,2.25);

\draw[axis]
  (-1.055,0) -- (2.35,0)
  node[right] {$\alpha$};

\draw[axis]
  (0,-0.065) -- (0,2.255)
  node[above right] {$\beta$};

\node[regionlabel]
  at (-0.55,1.15)
  {$\mathcal G$};

\node[regionlabel]
  at (0.30,1.55)
  {$\mathcal M_c^{+}$};

\node[regionlabel]
  at (1.25,1.55)
  {$\mathcal N_c$};

\draw[black!68]
  (-1,0.024) -- (-1,-0.024);

\node[
  font=\scriptsize,
  below,
  text=textgray
]
  at (-1,-0.028)
  {$-1$};

\draw[axisgray!60!black]
  (1,0.024) -- (1,-0.024);

\node[
  font=\scriptsize,
  below,
  text=textgray
]
  at (1,-0.028)
  {$1$};

\draw[axisgray!60!black]
  (2,0.024) -- (2,-0.024);

\node[
  font=\scriptsize,
  below,
  text=textgray
]
  at (2,-0.028)
  {$2$};

\draw[black!68]
  (0.024,1) -- (-0.024,1);

\node[
  font=\scriptsize,
  left,
  text=textgray
]
  at (-0.028,1)
  {$1$};

\draw[axisgray!60!black]
  (0.024,2) -- (-0.024,2);

\node[
  font=\scriptsize,
  left,
  text=textgray
]
  at (-0.028,2)
  {$2$};

\node[
  font=\footnotesize\bfseries,
  text=textgray,
  anchor=west
]
  at (-1.00,2.52)
  {\textup{(a)} $c=\tfrac12$};

\end{scope}
\end{tikzpicture}
\end{minipage}%
\hfill
\begin{minipage}{0.44\textwidth}
\centering
\begin{tikzpicture}[
  line cap=round,
  line join=round
]
\tikzset{
  axis/.style={
    draw=axisgray,
    line width=0.68pt,
    ->
  },
  boundary/.style={
    draw=black!72,
    line width=0.76pt
  },
  criterion/.style={
    draw=black!62,
    line width=0.70pt,
    dash pattern=on 3.6pt off 2.6pt
  },
  frontier/.style={
    draw=black!60,
    line width=0.70pt,
    dash pattern=on 4.0pt off 1.6pt on 0.8pt off 1.6pt
  },
  excluded/.style={
    draw=black!52,
    line width=0.58pt,
    dash pattern=on 2.0pt off 2.0pt
  },
  leader/.style={
    draw=black!52,
    line width=0.45pt
  },
  linelabel/.style={
    font=\scriptsize,
    text=textgray,
    inner sep=0pt
  },
  regiontext/.style={
    font=\scriptsize,
    text=textgray,
    inner sep=0pt,
    align=center
  }
}

\begin{scope}[x=4.20cm,y=4.20cm]

\pgfmathsetmacro{\rhoclower}{4*sqrt(5)-9}

\begin{scope}
\clip (-1.05,-1.10) rectangle (0.10,0.08);

\fill[regionm]
  (-1,0) --
  (0,0) --
  (-1,\rhoclower) --
  cycle;

\fill[regionnew]
  (-1,-0.5) --
  (-1,\rhoclower) --
  (0,0) --
  (-0.75,-0.75) --
  cycle;

\fill[regionopen]
  (-1,-1) --
  (-1,-0.5) --
  (-0.75,-0.75) --
  cycle;

\fill[
  pattern={
    Lines[
      angle=30,
      distance=5.0pt,
      line width=0.24pt
    ]
  },
  pattern color=black!22
]
  (-1,-1) --
  (-1,-0.5) --
  (-0.75,-0.75) --
  cycle;
\end{scope}

\draw[criterion]
  (-1,\rhoclower) -- (0,0)
  node[
    pos=0.58,
    sloped,
    below=2.2pt,
    linelabel
  ]
  {$\beta=\rho_c\alpha$};

\draw[frontier]
  (-1,-0.5) -- (-0.75,-0.75);

\draw[boundary]
  (-1,-1) -- (0,0)
  node[
    pos=0.76,
    sloped,
    below=2.0pt,
    linelabel
  ]
  {$\beta=\alpha$};

\draw[excluded]
  (-1,-1.03) -- (-1,0.03);

\draw[axis]
  (-1.055,0) -- (0.055,0)
  node[right] {$\alpha$};

\draw[axis]
  (0,-1.055) -- (0,0.055)
  node[above right] {$\beta$};

% External label and leader for the narrow region M_c^{-}.
\node[
  font=\scriptsize,
  text=textgray,
  anchor=south west,
  inner sep=0pt
]
  at (-0.48,0.065)
  {$\mathcal M_c^{-}$};

\draw[leader]
  (-0.43,0.055) --
  (-0.60,0.020) --
  (-0.86,-0.024);

\node[regiontext]
  at (-0.40,-0.23)
  {$\mathcal P_c$};

\node[regiontext]
  at (-0.87,-0.76)
  {$\mathcal O_c$};

\node[
  font=\tiny,
  text=black!38,
  align=center,
  inner sep=0pt
]
  at (-0.30,-0.74)
  {outside $\alpha<\beta$};

\draw[axisgray!60!black]
  (-1,0.012) -- (-1,-0.012);

\node[
  font=\scriptsize,
  anchor=north east,
  text=textgray
]
  at (-1.012,-0.014)
  {$-1$};

\draw[axisgray!60!black]
  (0.012,-1) -- (-0.012,-1);

\node[
  font=\scriptsize,
  left,
  text=textgray
]
  at (-0.015,-1)
  {$-1$};

\node[
  font=\footnotesize\bfseries,
  text=textgray,
  anchor=west
]
  at (-1.00,0.20)
  {\textup{(b)} $c=\tfrac{2}{\sqrt5}$};

\end{scope}
\end{tikzpicture}
\end{minipage}

\caption{Degree-uniform parameter regions within
$-1<\alpha<\beta$. Panel~\textup{(a)} displays the portion of the
admissible domain with $\beta\geq0$ for $c=1/2$, while
panel~\textup{(b)} displays the portion with $\beta<0$ for
$c=2/\sqrt5$. The regions $\mathcal G$ and $\mathcal M_c^\pm$ were
previously known; $\mathcal N_c$ is the positive-sector region and
$\mathcal P_c$ is the range of the Ward estimate. The hatched region
$\mathcal O_c$ is its complement in the negative wedge, covered by
Remark~\ref{rem:crossing-addendum}; it is empty when
$0<c\leq\sqrt3/2$. The internal dashed lines are criterion
boundaries. In panel~\textup{(b)}, the short dot-dashed segment is
$\alpha+\beta=-3/2$, the boundary of the Ward range for
the displayed value of $c$. In both panels, the uncoloured portion
below the line $\beta=\alpha$ lies outside the admissible domain.
The lines $\beta=\alpha$ and $\alpha=-1$ are excluded. The fading at
the upper edge of panel~\textup{(a)} indicates continuation beyond
the plotting window.}
\label{fig:parameter-range}
\end{figure}

Section~\ref{sec:boundary} derives the boundary identity and recovers
the two earlier regions. The positive-sector result is proved in
Section~\ref{sec:positive} by means of integration and root-motion
identities and the ensuing crossing lemma.
Section~\ref{sec:degree-one} treats degree one.
Section~\ref{sec:ensemble} gives the Ward bound and proves
Theorem~\ref{thm:global}, invoking the refinement in
Remark~\ref{rem:crossing-addendum} for $\mathcal O_c$.
Section~\ref{sec:asymptotic} treats the large-degree limit, and
Section~\ref{sec:mobius} gives a M\"obius reformulation of the
endpoint inequality.

\section{The boundary identity and the two earlier regions}
\label{sec:boundary}

Set
\[
\delta=\beta-\alpha>0,
\quad
s=\alpha+\beta,
\]
and write
\[
h_n=\int_{-c}^{c}\pi_n(x)^2w(x)\,dx.
\]
The endpoint difference relevant to the conjecture is
\begin{equation}
D_n=
w(c)\pi_n(c)^2-w(-c)\pi_n(-c)^2.
\label{eq:endpoint-difference}
\end{equation}
The calculation leading to the next identity is the core of
Milovanovi\'c's argument
\cite[Equation~(8) and the display following it]{Milovanovic2022};
we record the result directly on the original interval $[-c,c]$.

\begin{proposition}\label{prop:boundary}
Conjecture~\ref{conj:gautschi} is equivalent to $D_n>0$. Moreover,
\begin{equation}
D_n
=
\int_{-c}^{c}w'(x)\pi_n(x)^2\,dx
\label{eq:boundary-identity}
\end{equation}
and
\begin{equation}
\frac{D_n}{h_n}
=
\int_{-c}^{c}
\frac{\delta-sx}{1-x^2}\,d\eta_n(x),
\quad
d\eta_n(x)=
\frac{\pi_n(x)^2w(x)}{h_n}\,dx.
\label{eq:induced-identity}
\end{equation}
\end{proposition}

\begin{proof}
Since
\[
\frac{w(-c)}{w(c)}
=
\left(\frac{1-c}{1+c}\right)^\delta,
\]
the asserted equivalence follows from
\eqref{eq:endpoint-difference}. The endpoint
values of $\pi_n$ do not vanish because all its zeros lie in
$(-c,c)$.

The fundamental theorem of calculus gives
\[
\begin{split}
D_n
&=
\int_{-c}^{c}
\bigl(w(x)\pi_n(x)^2\bigr)'\,dx
\\[7pt]
&=
\int_{-c}^{c}w'(x)\pi_n(x)^2\,dx
+
2\int_{-c}^{c}w(x)\pi_n(x)\pi_n'(x)\,dx.
\end{split}
\]
The second integral is zero because $\pi_n'$ has degree $n-1$.
This proves \eqref{eq:boundary-identity}. Finally,
\[
\frac{w'(x)}{w(x)}
=
\frac{\delta-sx}{1-x^2},
\]
and \eqref{eq:induced-identity} follows.
\end{proof}

\begin{corollary}[The mixed-sign sector and its boundary]
\label{cor:gautschi-sector}
If $\alpha\leq0\leq\beta$, then
Conjecture~\ref{conj:gautschi} holds for every $n\geq1$.
\end{corollary}

\begin{proof}
For $-c<x<c$,
\[
\frac{w'(x)}{w(x)}
=
-\frac{\alpha}{1-x}+\frac{\beta}{1+x}>0.
\]
Hence $w'(x)>0$, and Proposition~\ref{prop:boundary} gives
\[
D_n=\int_{-c}^{c}w'(x)\pi_n(x)^2\,dx>0.
\]
\end{proof}

\begin{corollary}[Milovanovi\'c's criterion]
\label{cor:milovanovic-region}
If \eqref{eq:known-criterion} holds, then
Conjecture~\ref{conj:gautschi} holds for every $n\geq1$.
\end{corollary}

\begin{proof}
For $-c\leq x\leq c$,
\[
\delta-sx\geq\delta-c\lvert s\rvert\geq0.
\]
In the equality case, the numerator can vanish only at one endpoint,
and is therefore positive almost everywhere in $(-c,c)$. Since
$\pi_n(x)^2w(x)>0$ almost everywhere, the integral in
\eqref{eq:induced-identity} is strictly positive.
\end{proof}

\section{The positive sector}
\label{sec:positive}

\subsection{Integration and root motion}

Let $(p_n)_{n\geq0}$, where $p_n=\pi_n/\sqrt{h_n}$, be the
corresponding orthonormal polynomials, chosen with positive leading
coefficients, and write
\begin{equation}
xp_n(x)
=
a_{n+1}p_{n+1}(x)+b_np_n(x)+a_np_{n-1}(x),
\quad n\geq0,
\label{eq:three-term}
\end{equation}
where $p_{-1}=0$, $a_0=0$, and $a_n>0$ for $n\geq1$.
Let $\Sigma_n$ be the sum of the zeros of $p_n$, with
$\Sigma_0=0$. Writing \eqref{eq:three-term} in monic form and
comparing the coefficients of $x^n$ gives
\[
\Sigma_{n+1}=\Sigma_n+b_n.
\]
Set
\[
\lambda=\frac{\beta-\alpha}{2},
\quad
\nu=\frac{\alpha+\beta+2}{2},
\]
and define
\[
U_n=\sqrt{w(c)}\,p_n(c),
\quad
V_n=(-1)^n\sqrt{w(-c)}\,p_n(-c).
\]
Since all the zeros of $p_n$ lie in $(-c,c)$, we have
$p_n(c)>0$ and $(-1)^np_n(-c)>0$, hence $U_n,V_n>0$.
Finally, put
\begin{equation}
d_n=
\frac{1-c^2}{2}\bigl(U_n^2-V_n^2\bigr)
=
\frac{1-c^2}{2h_n}D_n.
\label{eq:normalized-endpoint}
\end{equation}
Thus $d_n>0$ is equivalent to the conjectured endpoint inequality.

\begin{lemma}\label{lem:integration-identity}
For every $n\geq0$,
\begin{equation}
d_n
=
\lambda-\Sigma_n-(n+\nu)b_n
=
\lambda+(n+\nu-1)\Sigma_n-(n+\nu)\Sigma_{n+1}.
\label{eq:integration-identity}
\end{equation}
\end{lemma}

\begin{proof}
Since
\[
\bigl((1-x^2)w(x)\bigr)'
=
\bigl(2\lambda-2\nu x\bigr)w(x),
\]
integration of
$\bigl((1-x^2)wp_n^2\bigr)'$ over $[-c,c]$ gives
\[
2d_n
=
2\lambda-2\nu b_n
+
2\int_{-c}^{c}(1-x^2)p_n(x)p_n'(x)w(x)\,dx.
\]
Orthogonality annihilates the part containing $p_np_n'$. To identify
the coefficient of $p_n$ in the orthogonal expansion of $x^2p_n'$,
write
\[
\pi_n(x)=x^n-\Sigma_nx^{n-1}+\cdots,
\quad
\pi_{n+1}(x)=x^{n+1}-\Sigma_{n+1}x^n+\cdots.
\]
Comparison of the coefficients of $x^{n+1}$ and $x^n$ gives
\[
x^2\pi_n'
=
n\pi_{n+1}+(\Sigma_n+nb_n)\pi_n
+\text{a polynomial of degree at most }n-1.
\]
Hence
\[
\int_{-c}^{c}(1-x^2)p_np_n'w\,dx
=
-\Sigma_n-nb_n,
\]
which proves the first equality in \eqref{eq:integration-identity}. The
second follows from $b_n=\Sigma_{n+1}-\Sigma_n$.
\end{proof}

\begin{corollary}\label{cor:recurrence-form}
For every $n\geq0$,
\begin{equation}
(1-c^2)\frac{D_n}{h_n}
=
\delta-2\sum_{j=0}^{n-1}b_j-(s+2n+2)b_n.
\label{eq:recurrence-form}
\end{equation}
Consequently, for $n\geq1$, the conjectured inequality is equivalent to
\begin{equation}
\sum_{j=0}^{n-1}b_j+
\left(n+1+\frac{s}{2}\right)b_n<\frac{\delta}{2}.
\label{eq:recurrence-criterion}
\end{equation}
\end{corollary}

\begin{proof}
Use $\Sigma_n=\sum_{j=0}^{n-1}b_j$ in the first equality of
\eqref{eq:integration-identity}, and then use
\eqref{eq:normalized-endpoint}.
\end{proof}

For every \(k\geq0\), the moment
\[
m_k(c)
=
\int_{-c}^{c}x^kw(x)\,dx
=
c^{k+1}
\int_{-1}^{1}
t^k(1-ct)^\alpha(1+ct)^\beta\,dt
\]
is real-analytic for \(0<c<1\). Writing
\[
\pi_n(x;c)
=
x^n+\sum_{j=0}^{n-1}q_{j,n}(c)x^j,
\]
the orthogonality conditions are equivalent to the Hankel system
\[
\sum_{j=0}^{n-1}
m_{j+k}(c)q_{j,n}(c)
=
-m_{n+k}(c),
\quad
0\leq k<n.
\]
Its coefficient matrix is positive definite, since it is the moment
matrix of a positive measure with infinite support, and hence its
determinant is strictly positive. Cramer's rule therefore shows that
the coefficients \(q_{j,n}(c)\), and consequently
\(\pi_n\), \(h_n\), \(\Sigma_n\), and \(b_n\), are real-analytic in
\(c\). Since
\[
a_n=\sqrt{\frac{h_n}{h_{n-1}}},
\]
the coefficients \(a_n\) are real-analytic as well. The same then
holds for \(U_n\), \(V_n\), and \(d_n\).

\begin{lemma}[Motion of the root sum]\label{lem:root-sum-motion}
For every $n\geq1$,
\begin{equation}
\partial_c\Sigma_n
=
a_n\bigl(U_nU_{n-1}-V_nV_{n-1}\bigr).
\label{eq:root-sum-motion}
\end{equation}
Here and below, $\partial_c$ denotes differentiation with respect to
$c$.
\end{lemma}

\begin{proof}
Since
\[
\pi_n(x)=x^n-\Sigma_nx^{n-1}+\cdots,
\]
differentiation of
\[
\int_{-c}^{c}\pi_n(x)\pi_{n-1}(x)w(x)\,dx=0
\]
with respect to $c$ gives
\[
(\partial_c\Sigma_n)h_{n-1}
=
w(c)\pi_n(c)\pi_{n-1}(c)
+
w(-c)\pi_n(-c)\pi_{n-1}(-c).
\]
Since all zeros lie in $(-c,c)$,
\[
\operatorname{sgn}\pi_j(-c)=(-1)^j.
\]
After orthonormalisation and the identity
$\sqrt{h_n/h_{n-1}}=a_n$, this becomes
\eqref{eq:root-sum-motion}.
\end{proof}

\subsection{The crossing lemma}

\begin{lemma}[Crossing lemma]\label{lem:crossing}
Let $-1<\alpha<\beta$ and $n\geq1$. Suppose that, at $c=c_0$,
\[
d_n(c_0)=0,\quad d_{n-1}(c_0)>0,\quad
\Sigma_n(c_0)<\lambda.
\]
Then $\partial_c d_n(c_0)>0$.
\end{lemma}

\begin{proof}
At $c_0$ one has $U_n=V_n=:X>0$. Hence
\eqref{eq:root-sum-motion} and $d_{n-1}>0$ give
\[
\partial_c\Sigma_n
=a_nX(U_{n-1}-V_{n-1})>0.
\]
The first equality in \eqref{eq:integration-identity} gives
\[
b_n=\frac{\lambda-\Sigma_n}{n+\nu}>0.
\]
Evaluation of \eqref{eq:three-term} at $\pm c_0$, with the signs
absorbed into the definition of $V_j$, gives
\[
\begin{aligned}
c_0X
&=a_{n+1}U_{n+1}+b_nX+a_nU_{n-1},
\\[7pt]
c_0X
&=a_{n+1}V_{n+1}-b_nX+a_nV_{n-1}.
\end{aligned}
\]
Subtracting these identities yields
\[
a_{n+1}(U_{n+1}-V_{n+1})
=-2b_nX-a_n(U_{n-1}-V_{n-1}).
\]
Applying \eqref{eq:root-sum-motion} at levels $n$ and $n+1$ gives
\[
\partial_c\Sigma_{n+1}
=-2b_nX^2-\partial_c\Sigma_n.
\]
Differentiation of the second equality in
\eqref{eq:integration-identity} now gives
\begin{equation}
\partial_c d_n(c_0)
=
(2n+2\nu-1)\partial_c\Sigma_n
+2(n+\nu)b_nX^2>0,
\label{eq:crossing-derivative}
\end{equation}
because $\nu>0$ and hence $2n+2\nu-1>0$.
\end{proof}

\begin{remark}[Addendum to the first-crossing argument]
\label{rem:crossing-addendum}
The conclusion of Lemma~\ref{lem:crossing} remains valid without the
hypothesis $\Sigma_n(c_0)<\lambda$.
Indeed, integrating $(\sigma wp_np_k)'$ for $k<n$, where $\sigma(x)=1-x^2$,
and using \eqref{eq:three-term}, we obtain, as in
Lemma~\ref{lem:integration-identity},
\begin{equation}\label{eq:addendum-expansion}
\begin{split}
 \sigma(x)p_n'(x)={}&-(nx+\Sigma_n)p_n(x)+C_na_np_{n-1}(x)\\[7pt]
 &+\sigma(c)\sum_{k<n}
   \bigl(U_nU_k-(-1)^{n+k}V_nV_k\bigr)p_k(x),
\end{split}
\end{equation}
where $C_n=2n+2\nu-1$ and primes denote differentiation in $x$.
At the crossing $c=c_0$, put $X=U_n=V_n>0$.
Divide by $p_n$, evaluate at $\pm c$ and add.
Markov's theorem \cite[Theorem~6.12.1]{Szego1975} and the telescopic
sum of~\eqref{eq:integration-identity} then give
\begin{equation}\label{eq:addendum-static-crossing}
 C_na_n\frac{U_{n-1}-V_{n-1}}{X}
 \geq 2(n+\nu)\Sigma_n-2n\lambda.
\end{equation}
Combining this with~\eqref{eq:integration-identity},
\eqref{eq:root-sum-motion} and the equality in
\eqref{eq:crossing-derivative}, we obtain
\[
 (n+\nu)\partial_cd_n
 \geq(n+\nu-1)C_n\partial_c\Sigma_n+2\nu\lambda X^2>0.
\]
Here $\partial_c\Sigma_n>0$ follows from
\eqref{eq:root-sum-motion} and $d_{n-1}>0$.
Since $n\geq1$ and $\nu,\lambda>0$, the inequality is strict.
\end{remark}

\subsection{The positive sector}

\begin{proposition}\label{prop:positive-sector}
Conjecture~\ref{conj:gautschi} holds for every $n\geq1$,
$0<c<1$, and $0<\alpha<\beta$.
\end{proposition}

\begin{proof}
Since $\nu>1$, we prove by induction on $n\geq0$ that
\begin{equation}
d_n(c)>0
\quad\text{and}\quad
\Sigma_{n+1}(c)<\lambda,
\quad 0<c<1.
\label{eq:induction-claim}
\end{equation}
For $n=0$,
\[
U_0^2-V_0^2
=
\frac{w(c)-w(-c)}{h_0}>0,
\quad
\frac{w(c)}{w(-c)}
=
\left(\frac{1+c}{1-c}\right)^{\beta-\alpha}>1.
\]
The identity \eqref{eq:integration-identity} with $n=0$ then gives
\(\nu\Sigma_1<\lambda\), and hence $\Sigma_1<\lambda$.

Assume \eqref{eq:induction-claim} at level $n-1$. Since
$s=\alpha+\beta>0$, Corollary~\ref{cor:milovanovic-region} gives
$d_n(c)>0$ for
\begin{equation}
0<c\leq c_*:=\frac{\beta-\alpha}{\alpha+\beta}<1.
\label{eq:c-star}
\end{equation}
If $d_n$ vanished to the right of this interval, let
$Z_n:=\{c\in(c_*,1):d_n(c)=0\}$. Since $d_n(c_*)>0$, continuity
gives $\varepsilon>0$ with $d_n>0$ on $[c_*,c_*+\varepsilon]$.
Choose any $c_1\in Z_n$. The set
$Z_n\cap[c_*+\varepsilon,c_1]$ is a non-empty closed subset of the
compact interval $[c_*+\varepsilon,c_1]$, and hence has a least
element, say $c_0$. By the definition of $c_0$, the function $d_n$ is
positive immediately to its left, and therefore
\[
\partial_c d_n(c_0)=
\lim_{h\downarrow0}
\frac{d_n(c_0)-d_n(c_0-h)}{h}
\leq0.
\]
The induction hypothesis gives $d_{n-1}(c_0)>0$ and
$\Sigma_n(c_0)<\lambda$, so Lemma~\ref{lem:crossing} gives
$\partial_c d_n(c_0)>0$, a contradiction. Thus $d_n(c)>0$ on
$(0,1)$.

It remains to propagate the root-sum bound. By
\eqref{eq:integration-identity},
\[
(n+\nu)\Sigma_{n+1}
<
\lambda+(n+\nu-1)\Sigma_n
<
(n+\nu)\lambda.
\]
Therefore $\Sigma_{n+1}<\lambda$, completing the induction.
\end{proof}

\section{The degree-one case}
\label{sec:degree-one}

\begin{lemma}\label{lem:degree-one}
Let $0<c<1$ and $-1<\alpha<\beta$.
Conjecture~\ref{conj:gautschi} holds for $n=1$.
\end{lemma}

\begin{proof}
Put
\[
\ell=\operatorname{arctanh}c,
\quad
\lambda=\frac{\beta-\alpha}{2},
\quad
\nu=\frac{\alpha+\beta+2}{2}.
\]
The parameter assumptions imply
\[
\nu-\lambda=\alpha+1>0,
\quad
\nu+\lambda=\beta+1>0.
\]
Make the change of variables
\[
x=\tanh u,
\quad
-\ell<u<\ell.
\]
Using
\[
1+\tanh u=\frac{e^u}{\cosh u},
\quad
1-\tanh u=\frac{e^{-u}}{\cosh u},
\]
we obtain
\begin{equation}
w(x)\,dx
=
e^{2\lambda u}\operatorname{sech}^{2\nu}u\,du.
\label{eq:hyperbolic-density}
\end{equation}
Let $T$ be a random variable with the probability density
proportional to the right-hand side of
\eqref{eq:hyperbolic-density}, and set
\[
\mu=\mathbb E(\tanh T).
\]
The monic polynomial of degree one is
\[
\pi_1(x)=x-\mu.
\]
Since $c-\mu>0$, $c+\mu>0$, and
\[
\frac{1-c}{1+c}=e^{-2\ell},
\]
taking the positive square root shows that the conjectured inequality
is equivalent to
\[
\frac{c+\mu}{c-\mu}<e^{2\lambda\ell},
\]
or, equivalently,
\begin{equation}
\mu<c\tanh(\lambda\ell).
\label{eq:degree-one-target}
\end{equation}

Let $R=\lvert T\rvert$. Conditional on $R=\tau$,
\[
\mathbb E(\tanh T\mid R=\tau)
=
\tanh \tau\,\tanh(2\lambda \tau)
=:H(\tau).
\]
The function $H$ is strictly increasing on $(0,\ell)$. The density
of $R$ is proportional to
\[
\operatorname{sech}^{2\nu}\tau
\cosh(2\lambda \tau)\,d\tau.
\]
Let $\sigma$ be the probability measure on $(0,\ell)$ with density
proportional to $\cosh(2\lambda \tau)$, and put
\[
g(\tau)=\operatorname{sech}^{2\nu}\tau.
\]
Then
\[
\mu=\frac{\mathbb E_{\sigma}(Hg)}
{\mathbb E_{\sigma}g}.
\]
The function \(g\) is strictly decreasing. If \(R_1\) and \(R_2\)
are independent random variables with law \(\sigma\), then
\[
\operatorname{Cov}_{\sigma}(H,g)
=
\frac12
\mathbb E_{\sigma}
\bigl[
(H(R_1)-H(R_2))
(g(R_1)-g(R_2))
\bigr]
\leq0.
\]
The inequality is strict because \(\sigma\) has a strictly positive
density on \((0,\ell)\), while \(H\) and \(g\) are strictly monotone
in opposite directions. Thus the product inside the expectation is
strictly negative whenever \(R_1\neq R_2\), and the diagonal has
\(\sigma\otimes\sigma\)-measure zero. Hence
\[
\operatorname{Cov}_{\sigma}(H,g)<0.
\]
Consequently,
\begin{equation}
\mu<\mathbb E_{\sigma}H=:m.
\label{eq:mu-comparison}
\end{equation}
The auxiliary measure $\sigma$ need not correspond to admissible
Jacobi parameters; only its finiteness on the compact interval is
used.

By the definition of $H$,
\[
m=
\frac{
\displaystyle
\int_0^\ell\tanh u\,\sinh(2\lambda u)\,du
}{
\displaystyle
\int_0^\ell\cosh(2\lambda u)\,du
}.
\]
Write $z=\lambda\ell$. Integration by parts gives
\begin{equation}
\begin{split}
2\lambda\int_0^\ell\tanh u\,\sinh(2\lambda u)\,du
&=
c\cosh(2z)
-
\int_0^\ell
\operatorname{sech}^2u\,\cosh(2\lambda u)\,du
\\[7pt]
&<
c\bigl(\cosh(2z)-1\bigr),
\end{split}
\label{eq:degree-one-ibp}
\end{equation}
because $\cosh(2\lambda u)>1$ for $u>0$ and
\[
\int_0^\ell\operatorname{sech}^2u\,du=c.
\]
Furthermore,
\[
2\lambda\int_0^\ell\cosh(2\lambda u)\,du=\sinh(2z)
\]
and
\[
\frac{\cosh(2z)-1}{\sinh(2z)}=\tanh z.
\]
It follows from \eqref{eq:degree-one-ibp} that
\[
m<c\tanh z=c\tanh(\lambda\ell).
\]
Together with \eqref{eq:mu-comparison}, this proves
\eqref{eq:degree-one-target}.
\end{proof}

The first-moment reduction is specific to degree one, since $\pi_1$
is determined by $\mu$. For $n\geq2$, orthogonality involves moments
up to order $2n-1$, so the same covariance comparison does not
directly control the endpoint ratio. Proposition~\ref{prop:mobius}
gives a general moment formulation, but its induced measure contains
$Q_n^2$, which requires additional control. The full proof uses the
present lemma as its base case and the refined crossing argument of
Remark~\ref{rem:crossing-addendum} for the induction.

\section{An ensemble bound in the residual sector}
\label{sec:ensemble}

The residual negative sector $\mathcal P_c\cup\mathcal O_c$ admits
the parametrisation
\begin{equation}
\alpha=-r-\lambda,
\quad
\beta=-r+\lambda.
\label{eq:residual-parametrisation}
\end{equation}
In this remaining negative wedge, the parameter restrictions are
\begin{equation}
0<r<1,
\quad
0<\lambda<\min\{r,1-r\},
\quad
\lambda<rc.
\label{eq:residual-parameter-range}
\end{equation}
In this notation,
\[
w(x)
=(1-x^2)^{-r}
\exp\bigl(2\lambda\operatorname{arctanh}x\bigr).
\]

We first extract a root-sum estimate from the associated orthogonal
polynomial ensemble. For the duration of the next lemma, the endpoint
of the interval is denoted by $a$ in order to reserve $c$ for the
endpoint fixed in Theorem~\ref{thm:global}.

\begin{lemma}[Ensemble Ward bound]
\label{lem:ensemble-ward}
Let $0<r<1$, $\lambda>0$, and $0<a<1$. For the weight
\[
w_\lambda(x)
=(1-x^2)^{-r}
\exp\bigl(2\lambda\operatorname{arctanh}x\bigr)
\]
on $[-a,a]$, let $\Sigma_N$ be the sum of the zeros of the monic
orthogonal polynomial of degree $N$. Then, for every $N\geq1$,
\begin{equation}
(N-ra^2)\Sigma_N<N\lambda a^2.
\label{eq:cumulative-root-sum}
\end{equation}
\end{lemma}

\begin{proof}
Let
\[
\Omega_{N,a}:=
\left\{
\boldsymbol{x}\in\mathbb R^N:\,
-a<x_1<\cdots<x_N<a
\right\}
\]
be the ordered chamber, and write
\[
\Delta(\boldsymbol{x})
=
\prod_{1\leq i<j\leq N}(x_j-x_i)
\]
for the Vandermonde determinant. Consider the probability measure
\begin{equation}
d\mathbb P_{N,t}(\boldsymbol{x})
=
\frac{1}{Z_{N,t}}
\Delta(\boldsymbol{x})^2
\prod_{i=1}^N
(1-x_i^2)^{-r}
e^{2t\operatorname{arctanh}x_i}\,d\boldsymbol{x}.
\label{eq:tilted-ope}
\end{equation}
The measure $\mathbb P_{N,0}$ is multivariate totally positive of
order two (MTP$_2$). Indeed, the ordered chamber is a sublattice under
coordinatewise minimum and maximum. If
$\boldsymbol{m}=\boldsymbol{x}\wedge\boldsymbol{y}$ and
$\boldsymbol{M}=\boldsymbol{x}\vee\boldsymbol{y}$, then, for
$i<j$,
\[
(m_j-m_i)(M_j-M_i)
\geq
(x_j-x_i)(y_j-y_i).
\]
There is equality when the two coordinate comparisons have the same
orientation. In the crossed case $x_i\leq y_i$ and $x_j\geq y_j$,
the difference between the left- and right-hand sides is
\[
(y_i-x_i)(x_j-y_j)\geq0;
\]
the other crossed case is identical. Squaring and multiplying over
$i<j$ proves the MTP$_2$ lattice inequality for
$\Delta^2$. The one-particle factors cancel under coordinatewise
minimum and maximum. Extending the density by zero outside the
ordered chamber preserves the same inequality because the chamber is
a sublattice. Thus it defines an MTP$_2$ density on
$\mathbb R^N$.

Since $a<1$, the one-particle factor in \eqref{eq:tilted-ope} is
continuously differentiable and bounded above and below by positive
constants on $[-a,a]$. Thus the density in \eqref{eq:tilted-ope},
extended by zero outside the ordered chamber, is a non-negative integrable
MTP$_2$ density on $\mathbb R^N$. After normalisation it is a
probability law, and the MTP$_2$ association theorem applies to bounded
coordinatewise increasing functions. More precisely, we use
\cite[Theorem~4.2]{KarlinRinott1980}. Put
\[
H(\boldsymbol{x})
=\sum_{i=1}^N\operatorname{arctanh}x_i,
\quad
S(\boldsymbol{x})
=\sum_{i=1}^N\frac{x_i^3}{1-x_i^2}.
\]
Both $H$ and $S$ are increasing in every coordinate; for $S$ this
follows from
\[
\left(\frac{x^3}{1-x^2}\right)'
=\frac{x^2(3-x^2)}{(1-x^2)^2}\geq0.
\]
The involution
\[
\iota(x_1,\ldots,x_N)=(-x_N,\ldots,-x_1)
\]
preserves both the ordered chamber and $\mathbb P_{N,0}$, and it
sends $H$ and $S$ to their negatives. Hence
$\mathbb E_{N,0}S=0$. Since $0<a<1$, all coordinates in
$\Omega_{N,a}$ satisfy $-a<x_i<a$, and therefore $H$, $S$ and
$e^{2\lambda H}$ are bounded. Hence the covariance terms appearing
below are finite. Moreover, \(e^{2\lambda H}\) is coordinatewise increasing. Since
\(\mathbb P_{N,\lambda}\) is the exponential tilt of
\(\mathbb P_{N,0}\) by this function, we have
\[
\mathbb E_{N,\lambda}S
=
\frac{
\mathbb E_{N,0}\bigl(Se^{2\lambda H}\bigr)
}{
\mathbb E_{N,0}e^{2\lambda H}
}.
\]
Both \(S\) and \(e^{2\lambda H}\) are bounded and coordinatewise
increasing. Hence the MTP$_2$ association theorem gives
\[
\operatorname{Cov}_{N,0}
\bigl(S,e^{2\lambda H}\bigr)
\geq0.
\]
Since \(\mathbb E_{N,0}S=0\), it follows that
\begin{equation}
\begin{split}
\mathbb E_{N,\lambda}S
&=
\frac{
\mathbb E_{N,0}\bigl(Se^{2\lambda H}\bigr)
}{
\mathbb E_{N,0}e^{2\lambda H}
}
\\[7pt]
&=
\frac{
\operatorname{Cov}_{N,0}
\bigl(S,e^{2\lambda H}\bigr)
}{
\mathbb E_{N,0}e^{2\lambda H}
}
\geq0.
\end{split}
\label{eq:fkg-S}
\end{equation}

We now use a Ward identity. Write $f_{N,\lambda}$ for the
density in \eqref{eq:tilted-ope}. For $\varepsilon>0$, define
\[
\Omega_{N,a}^{\varepsilon}:=
\left\{\boldsymbol{x}\in\Omega_{N,a}:\;
\begin{aligned}
-a+\varepsilon<x_1<\cdots<x_N<a-\varepsilon,\;
\\[7pt]
x_{i+1}-x_i>\varepsilon,\ i=1,\dots,N-1
\end{aligned}
\right\}.
\]
Applying the divergence theorem to the vector field
\[
F_i(\boldsymbol{x})=(a^2-x_i^2)f_{N,\lambda}(\boldsymbol{x}),\quad i=1,\dots,N,
\]
on $\Omega_{N,a}^{\varepsilon}$ gives
\[
\int_{\Omega_{N,a}^{\varepsilon}}\sum_{i=1}^N
\partial_{x_i}F_i\,d\boldsymbol{x}
=\int_{\partial\Omega_{N,a}^{\varepsilon}}F\cdot\nu\,dS.
\]
On endpoint faces, \(a^2-x_i^2=O(\varepsilon)\), so the
corresponding flux is \(O(\varepsilon)\). On collision faces,
\(\Delta(\boldsymbol{x})^2=O(\varepsilon^2)\), so the corresponding
flux is \(O(\varepsilon^2)\). Since there are only finitely many
faces and their \((N-1)\)-dimensional measures remain uniformly
bounded, the boundary flux tends to zero as
\(\varepsilon\downarrow0\).

It remains to justify passage to the full chamber in the volume
integral. All one-particle factors and their logarithmic derivatives
are bounded on \([-a,a]\). For \(i<j\), write
\[
\Delta(\boldsymbol{x})^2
=
(x_i-x_j)^2\Delta_{ij}(\boldsymbol{x})^2,
\]
where
\[
\Delta_{ij}(\boldsymbol{x})
=
\prod_{\substack{1\leq k<\ell\leq N\\[7pt]
\{k,\ell\}\neq\{i,j\}}}
(x_\ell-x_k).
\]
Upon expanding \(\partial_{x_i}F_i\), the only apparent singularities
are terms bounded by a constant multiple of
\[
\frac{\Delta(\boldsymbol{x})^2}{|x_i-x_j|}
=
|x_i-x_j|\Delta_{ij}(\boldsymbol{x})^2.
\]
This expression extends continuously, and hence boundedly, to the
closed cube \([-a,a]^N\). All the remaining terms are bounded there
as well. Consequently,
\[
\sum_{i=1}^N\partial_{x_i}F_i
\]
is dominated on \(\Omega_{N,a}\) by an integrable function independent
of \(\varepsilon\). Dominated convergence therefore gives
\[
\lim_{\varepsilon\downarrow0}
\int_{\Omega_{N,a}^{\varepsilon}}
\sum_{i=1}^N\partial_{x_i}F_i\,d\boldsymbol{x}
=
\int_{\Omega_{N,a}}
\sum_{i=1}^N\partial_{x_i}F_i\,d\boldsymbol{x}.
\]
Together with the vanishing boundary flux, this yields
\[
\int_{\Omega_{N,a}}
\sum_{i=1}^N
\partial_{x_i}
\left((a^2-x_i^2)f_{N,\lambda}\right)
\,d\boldsymbol{x}
=0.
\]
Expanding the derivative gives
\[
  \begin{split}
  0={}&
  -2\,\mathbb E_{N,\lambda}\sum_{i=1}^Nx_i
\\[7pt]
&+
2\,\mathbb E_{N,\lambda}
\sum_{i<j}
\frac{(a^2-x_i^2)-(a^2-x_j^2)}{x_i-x_j}
\\[7pt]
&+
2\,\mathbb E_{N,\lambda}\sum_{i=1}^N
  \frac{(a^2-x_i^2)(\lambda+rx_i)}{1-x_i^2}.
  \end{split}
  \]
The sum inside the middle expectation is
$-(N-1)\sum_i x_i$. All the statistics used here are symmetric, so their expectations
under the ordered-chamber law agree with those under the usual
unordered orthogonal polynomial ensemble. The standard determinantal
representation of such an ensemble gives the one-point intensity
\[
K_N(x,x)w_\lambda(x),
\quad
K_N(x,y)=\sum_{k=0}^{N-1}p_k(x)p_k(y);
\]
see, for example,
\cite[Lemma~2.8 and equations~(2.33)--(2.35)]{Konig2005}.
Hence
\[
\begin{split}
\mathbb E_{N,\lambda}\sum_{i=1}^Nx_i
&=
\int_{-a}^{a}xK_N(x,x)w_\lambda(x)\,dx
\\[7pt]
&=
\sum_{k=0}^{N-1}
\int_{-a}^{a}x p_k(x)^2w_\lambda(x)\,dx
\\[7pt]
&=
\sum_{k=0}^{N-1}b_k
=
\Sigma_N.
\end{split}
\]
Consequently,
\begin{equation}
N\Sigma_N
=
\mathbb E_{N,\lambda}\sum_{i=1}^N
\frac{(a^2-x_i^2)(\lambda+rx_i)}{1-x_i^2}.
\label{eq:ward-root-sum}
\end{equation}
Combining \eqref{eq:ward-root-sum} with
\[
\frac{(a^2-x^2)(\lambda+rx)}{1-x^2}
=
a^2(\lambda+rx)
-
(1-a^2)
\left(
\lambda\frac{x^2}{1-x^2}
+
r\frac{x^3}{1-x^2}
\right),
\]
and using
\[
\mathbb E_{N,\lambda}\sum_{i=1}^N x_i=\Sigma_N,
\]
together with the definition
\[
S=\sum_{i=1}^N\frac{x_i^3}{1-x_i^2},
\]
gives the exact identity
\begin{equation}
\begin{split}
(N-ra^2)\Sigma_N
={}&N\lambda a^2
\\[7pt]
&-(1-a^2)
\left\{
\lambda\,\mathbb E_{N,\lambda}
\sum_{i=1}^N\frac{x_i^2}{1-x_i^2}
+r\,\mathbb E_{N,\lambda}S
\right\}.
\end{split}
\label{eq:ward-exact}
\end{equation}
The first expectation in braces is strictly positive, and the second
is non-negative by \eqref{eq:fkg-S}. Since $N-ra^2>0$,
\eqref{eq:cumulative-root-sum} follows.
\end{proof}

\begin{proof}[Proof of Theorem~\ref{thm:global}]
The case $c=1$ is immediate. Let $0<c<1$.
Corollaries~\ref{cor:gautschi-sector} and~\ref{cor:milovanovic-region}
cover $\mathcal G\cup\mathcal M_c$, and
Proposition~\ref{prop:positive-sector} covers $\mathcal N_c$.
It remains to treat $\mathcal P_c\cup\mathcal O_c$.
Fix $c$ and the parameters in
\eqref{eq:residual-parametrisation}, and let the endpoint $a$ vary in
$(0,c]$. Write $d_n(a)$ for the quantity in
\eqref{eq:normalized-endpoint} with $c$ replaced by $a$. For
\[
0<a\leq a_*:=\frac{\lambda}{r},
\]
one has
\[
\lambda+rx\geq\lambda-ra\geq0,
\quad -a\leq x\leq a.
\]
The boundary identity therefore gives $d_n(a)>0$ for every degree.

We now argue by induction on $n$. The case $n=1$ is
Lemma~\ref{lem:degree-one}. Suppose that $n\geq2$ and
that $d_{n-1}(a)>0$ for every $0<a\leq c$. If $d_n$ vanished in this
interval, define
$Z_n:=\{a\in(a_*,c]:d_n(a)=0\}$. Since $d_n>0$ on
$(0,a_*]$, the same compactness argument as in the proof of
Proposition~\ref{prop:positive-sector} shows that $Z_n$ has a least
element $a_0>a_*$. Again,
\[
\partial_a d_n(a_0)\leq0.
\]
If $(\alpha,\beta)\in\mathcal P_c$, then
\[
a_0^2\leq c^2\leq\frac{3}{3+r}
\leq\frac{n+1}{n+1+r}.
\]
Lemma~\ref{lem:ensemble-ward} gives $\Sigma_{n+1}(a_0)<\lambda$.
Since $d_n(a_0)=0$, \eqref{eq:integration-identity} yields
\[
(n+\nu-1)(\Sigma_n(a_0)-\lambda)
=(n+\nu)(\Sigma_{n+1}(a_0)-\lambda)<0.
\]
Thus $\Sigma_n(a_0)<\lambda$, and Lemma~\ref{lem:crossing} gives
$\partial_a d_n(a_0)>0$.
If $(\alpha,\beta)\in\mathcal O_c$, the refinement in
Remark~\ref{rem:crossing-addendum} gives the same strict inequality
without a condition on $\Sigma_n(a_0)$.
Both cases contradict $\partial_a d_n(a_0)\leq0$.
Hence $d_n(a)>0$ on $(0,c]$, and the induction is complete.
\end{proof}

\begin{remark}[The range of the Ward estimate]
\label{rem:ensemble-limitation}
The set $\mathcal P_c$ records the range obtained by combining the
Ward estimate with Lemma~\ref{lem:crossing} before strengthening its
hypotheses. Define
\[
C_{N,a}
=
(1-a^2)
\left\{
\lambda\,\mathbb E_{N,\lambda}
\sum_{i=1}^{N}\frac{x_i^2}{1-x_i^2}
+
r\,\mathbb E_{N,\lambda}
\sum_{i=1}^{N}\frac{x_i^3}{1-x_i^2}
\right\}.
\]
The proof of Lemma~\ref{lem:ensemble-ward} establishes that
$C_{N,a}>0$, while \eqref{eq:ward-exact} reads
\[
(N-ra^2)\Sigma_N=N\lambda a^2-C_{N,a}.
\]
Consequently, the sign $C_{N,a}>0$ alone implies
$\Sigma_N<\lambda$ whenever
\[
a^2\leq\frac{N}{N+r}.
\]
For $N=n+1\geq3$, this holds uniformly under the condition
\begin{equation}
c^2\leq\frac{3}{3+r}.
\label{eq:residual-slice-condition}
\end{equation}
Since $r=-(\alpha+\beta)/2$, this is precisely the condition defining
$\mathcal P_c$ used in the proof of Theorem~\ref{thm:global}.

Beyond this range, positivity of $C_{N,a}$ alone does not imply the
root-sum inequality required by Lemma~\ref{lem:crossing}.
That route would require the quantitative estimate
\[
C_{N,a}
>
\lambda\bigl((N+r)a^2-N\bigr)
\]
at the corresponding putative crossings whenever the right-hand side
is positive, or an independent treatment of the initial degrees.
Remark~\ref{rem:crossing-addendum} removes the root-sum hypothesis
from the crossing lemma and therefore avoids this additional
quantitative requirement.
\end{remark}

\section{Asymptotic endpoint difference}
\label{sec:asymptotic}

\begin{proposition}\label{prop:asymptotic-limit}
For every fixed $0<c<1$ and $-1<\alpha<\beta$,
\begin{equation}
\lim_{n\to\infty}\frac{D_n}{h_n}
=\frac{\beta-\alpha}{\sqrt{1-c^2}}.
\label{eq:asymptotic-limit}
\end{equation}
\end{proposition}

\begin{proof}
Since $w$ is continuous and strictly positive on $[-c,c]$,
Rakhmanov's theorem \cite{Rakhmanov1977}, after scaling to $[-1,1]$,
gives $a_n\to c/2$ and $b_n\to0$.
The classical weak-convergence result
\cite[Section~4]{VanAssche1995} therefore yields
\[
p_n(x)^2w(x)\,dx
\xrightarrow{\mathrm w}
\frac{dx}{\pi\sqrt{c^2-x^2}}.
\]
The function $(\delta-sx)/(1-x^2)$ is continuous on $[-c,c]$.
Hence \eqref{eq:induced-identity} gives
\[
\frac{D_n}{h_n}\longrightarrow
\frac1\pi\int_{-c}^{c}
\frac{\delta-sx}{(1-x^2)\sqrt{c^2-x^2}}\,dx.
\]
The term containing $s$ is odd, and
\[
\frac1\pi\int_{-c}^{c}
\frac{dx}{(1-x^2)\sqrt{c^2-x^2}}
=\frac1{\sqrt{1-c^2}}.
\]
Since $\delta=\beta-\alpha$, this proves
\eqref{eq:asymptotic-limit}.
\end{proof}

\section{A M\"obius formulation of the endpoint inequality}
\label{sec:mobius}

Corollary~\ref{cor:recurrence-form} already expresses the conjecture
as a bound on recurrence coefficients. A change of variables gives a
second formulation in terms of a reciprocal moment of an induced
probability measure.

Put
\[
y_-=\frac{1-c}{1+c},
\quad
y_+=y_-^{-1},
\quad
y=\frac{1-x}{1+x}.
\]
Choose $\kappa_n\neq0$ so that
\begin{equation}
Q_n(y)
=
\kappa_n(1+y)^n
\pi_n\left(\frac{1-y}{1+y}\right)
\label{eq:mobius-polynomial}
\end{equation}
is monic. This is possible because $\pi_n(-1)\neq0$. Equivalently,
\[
\pi_n\left(\frac{1-y}{1+y}\right)
=
\kappa_n^{-1}\frac{Q_n(y)}{(1+y)^n}.
\]

\begin{proposition}\label{prop:mobius}
The polynomial $Q_n$ is orthogonal on $[y_-,y_+]$ with respect to
\begin{equation}
v_n(y)
=
y^\alpha(1+y)^{-\gamma},
\quad
\gamma=2n+\alpha+\beta+1.
\label{eq:mobius-weight}
\end{equation}
Put
\begin{equation}
M_n
=
\frac{
\displaystyle
\int_{y_-}^{y_+} Q_n(y)^2v_n(y)\,\frac{dy}{y}
}{
\displaystyle
\int_{y_-}^{y_+} Q_n(y)^2v_n(y)\,dy
}.
\label{eq:mobius-mean}
\end{equation}
Thus $M_n$ is the reciprocal moment of the probability measure
obtained by normalising $Q_n(y)^2v_n(y)\,dy$. Conjecture~\ref{conj:gautschi} is equivalent to
\begin{equation}
\alpha M_n<\beta.
\label{eq:mobius-criterion}
\end{equation}
\end{proposition}

\begin{proof}
Every polynomial in $x$ of degree at most $n-1$ becomes
$R(y)/(1+y)^{n-1}$, with $\deg R\leq n-1$, and the inverse
substitution shows that this correspondence is a bijection. Since
\[
dx=-\frac{2\,dy}{(1+y)^2}
\]
and
\[
w\left(\frac{1-y}{1+y}\right)
=
2^{\alpha+\beta}
y^\alpha(1+y)^{-\alpha-\beta},
\]
the orthogonality conditions transform exactly into those for the
weight \eqref{eq:mobius-weight}.

Let
\[
h=\int_{y_-}^{y_+} Q_n(y)^2v_n(y)\,dy.
\]
Integration by parts gives
\begin{equation}
\begin{split}
\left[(1+y)v_n(y)Q_n(y)^2\right]_{y_-}^{y_+}
&=
\int_{y_-}^{y_+}
\bigl((1+y)v_n Q_n^2\bigr)'\,dy
\\[7pt]
&=
h(\alpha M_n-\beta).
\end{split}
\label{eq:mobius-boundary}
\end{equation}
Here
\[
\int_{y_-}^{y_+} Q_nQ_n'v_n\,dy=0,
\quad
\int_{y_-}^{y_+} yQ_nQ_n'v_n\,dy=nh,
\]
the latter identity following from the leading coefficient of
$yQ_n'$. A direct substitution using $y_+=y_-^{-1}$ shows that
\begin{equation}
\frac{
(1+y_+)v_n(y_+)Q_n(y_+)^2
}{
(1+y_-)v_n(y_-)Q_n(y_-)^2
}
=
\left[
\frac{\pi_n(-c)}{\pi_n(c)}
\right]^2
\left(\frac{1-c}{1+c}\right)^{\beta-\alpha}.
\label{eq:mobius-ratio}
\end{equation}
Thus \eqref{eq:gautschi} is equivalent to the negativity of the
left-hand side of \eqref{eq:mobius-boundary}, which is precisely
\eqref{eq:mobius-criterion}.
\end{proof}

\begin{remark}
The reduction \eqref{eq:mobius-criterion} displays the two same-sign
sectors in opposite forms. Proposition~\ref{prop:positive-sector} proves
$M_n<\beta/\alpha$ when $0<\alpha<\beta$.
Theorem~\ref{thm:global} gives $M_n>\beta/\alpha$
throughout the negative sector. In the mixed-sign sector,
\eqref{eq:mobius-criterion} is
immediate, in agreement with
Corollary~\ref{cor:gautschi-sector}.
\end{remark}

\section{Concluding remarks}

Theorem~\ref{thm:global} covers the full admissible domain
\[
\mathcal G\cup\mathcal M_c\cup\mathcal N_c\cup\mathcal P_c
\cup\mathcal O_c.
\]
The extension to $\mathcal O_c$ follows from the strengthened
crossing lemma in Remark~\ref{rem:crossing-addendum}. It uses an
extension of the coefficient calculation in the proof of
Lemma~\ref{lem:integration-identity} and Markov's theorem on the
monotonicity of zeros; the first-crossing induction is unchanged.

The endpoint inequality retains two exact equivalent formulations:
the recurrence-coefficient inequality
\eqref{eq:recurrence-criterion} and, in the negative sector, the
reciprocal-moment inequality
\begin{equation}
M_n(\alpha,\beta;c)>\frac{\beta}{\alpha}.
\label{eq:remaining-mobius-problem}
\end{equation}
The Ward estimate identifies the explicit region $\mathcal P_c$,
while the direct degree-one proof and the large-degree limit remain
complementary descriptions of the same endpoint comparison.

\appendix
\section{Reflection inequalities and extensions for truncated Jacobi weights}
\label{app:reflection-extensions}

This appendix develops the AI-generated reflection approach communicated by Jonas Matuzas in response to version~2 of our paper, \href{https://arxiv.org/abs/2608.05963v2}{arXiv:2608.05963v2} . Matuzas, who is not a specialist in the theory of orthogonal polynomials, explicitly disclosed that the problem selection, proof search, symbolic and numerical auditing, and drafting of the communicated note were carried out by OpenAI’s ChatGPT, model GPT-5.6 Sol Pro, with prompts and supervision supplied by him. This approach is particularly effective for extending the parameter range and deriving quantitative bounds for the zeros. We nevertheless retain our original presentation because the differential identities, recurrence relations, and probabilistic interpretation provide complementary insights beyond the endpoint inequality itself. Keeping the alternative approach in this appendix also preserves a clear distinction between the original arguments and the developments arising from the subsequent exchange with Matuzas. We present the sufficient region and quantitative estimates in the communicated note, incorporating the corrections and clarifications introduced during our traditional review of the arguments. A separate treatment of degree one then yields a further enlargement of the parameter region.

Only in this appendix, we allow arbitrary real exponents
$\alpha<\beta$ and retain $0<c<1$. The restricted weight
$(1-x)^\alpha(1+x)^\beta$ is positive and bounded on $[-c,c]$,
so the orthogonal polynomials remain well defined even when an
exponent is at most $-1$. We keep the notation of the main text and put
\[
\lambda=\frac{\beta-\alpha}{2}>0,
\quad
r=-\frac{\alpha+\beta}{2},
\quad
\nu=1-r.
\]

\subsection{The parameter extension}

\begin{theorem}[Sufficient region communicated by Matuzas]
\label{thm:app-matuzas}
Let $0<c<1$ and let $\alpha<\beta$ be real. If
\begin{equation}
\alpha+\beta\geq\max\left\{-3,-\frac{2}{c^2}\right\},
\label{eq:app-matuzas-region}
\end{equation}
then, for every $n\geq1$,
\begin{equation}
\left[\frac{\pi_n(-c;c)}{\pi_n(c;c)}\right]^2
\left(\frac{1-c}{1+c}\right)^{\beta-\alpha}<1.
\label{eq:app-endpoint}
\end{equation}
\end{theorem}

The reflection estimate below complements the first-crossing
argument of Section~\ref{sec:positive}. Using the same idea, the
following refinement is obtained by separating the initial degree
from the induction.

\begin{theorem}
\label{thm:app-refinement}
The conclusion of Theorem~\ref{thm:app-matuzas} remains valid under
the weaker condition
\begin{equation}
\alpha+\beta\geq\max\left\{-5,-\frac{2}{c^2}\right\}.
\label{eq:app-refined-region}
\end{equation}
Both inequalities defining this region may be equalities.
\end{theorem}

The two theorems will be proved together after the reflection and
crossing identities. The region is sufficient; no assertion of
optimality is made outside the classical range.

\subsection{Reflection and exterior estimates}

Let $\mu_0$ be a finite positive even Borel measure with infinite
support contained in $[-a,a]$, where $a>0$. Let $h$ be continuous,
odd and strictly increasing on this interval, and define
\[
d\mu_\tau(x)=e^{2\tau h(x)}\,d\mu_0(x),
\quad \tau\in\mathbb R.
\]
Write $P_n(\,\cdot\,;\tau)$ for the monic orthogonal polynomial and
$x_{1,n}(\tau)<\cdots<x_{n,n}(\tau)$ for its zeros.

\begin{lemma}
\label{lem:app-reflection}
Every zero is strictly increasing in $\tau$, and
\[
x_{j,n}(-\tau)=-x_{n+1-j,n}(\tau),
\quad 1\leq j\leq n.
\]
Consequently, for $\tau>0$,
\begin{equation}
x_{j,n}(\tau)>-x_{n+1-j,n}(\tau).
\label{eq:app-reflection}
\end{equation}
For $\tau>0$ and every nondecreasing function $\varphi$ on $[-a,a]$,
\[
\sum_{j=1}^n\varphi(x_{j,n})
\geq\sum_{j=1}^n\varphi(-x_{j,n}).
\]
This inequality is strict if $\varphi$ is strictly increasing
throughout $[-a,a]$.
\end{lemma}

\begin{proof}
The zero motion is the classical theorem of Markov; see
\cite[Theorem~6.12.1]{Szego1975}. For completeness, its quadrature
proof also applies to the present measure. Put
$q_j(x)=P_n(x)/(x-x_{j,n})$, and let $A_j>0$ be the corresponding
Gaussian quadrature weight. Differentiating orthogonality gives
\[
\begin{split}
A_jP_n'(x_{j,n})^2\partial_\tau x_{j,n}
={}&2\int_{-a}^a
\bigl(h(x)-h(x_{j,n})\bigr)(x-x_{j,n})q_j(x)^2
\,d\mu_\tau(x)>0.
\end{split}
\]
The integrand is nonnegative and is positive outside a finite set;
the measure has infinite support. The required differentiability
follows from the analyticity of the moments in $\tau$ and the
simplicity of the zeros. Evenness of $\mu_0$ and oddness of $h$ give
$P_n(x;-\tau)=(-1)^nP_n(-x;\tau)$.
Strict motion from $-\tau$ to $\tau$ proves
\eqref{eq:app-reflection}. Applying $\varphi$ to these ordered pairs
and summing proves the remaining assertions.
\end{proof}

For fixed $\tau>0$, suppress this parameter and put
\[
\Sigma_n=\sum_{j=1}^n x_{j,n},
\quad
S_{2m+1}=\sum_{j=1}^n x_{j,n}^{2m+1},
\quad m\geq0.
\]
In particular, $\Sigma_n>0$.

\begin{proposition}
\label{prop:app-exterior}
For every integer $m\geq0$,
\begin{equation}
S_{2m+1}\geq\frac{\Sigma_n^{2m+1}}{n^{2m}}>0.
\label{eq:app-odd-moments}
\end{equation}
The first inequality is strict when $n>1$ and $m\geq1$.
For $t\geq a$, define
\[
R_n(t)=(-1)^n\frac{P_n(-t)}{P_n(t)},
\quad
\Lambda_n(t)=\frac{P_n'(t)}{P_n(t)}
+\frac{P_n'(-t)}{P_n(-t)}.
\]
Then
\begin{equation}
R_n(t)\geq
\left(\frac{nt+\Sigma_n}{nt-\Sigma_n}\right)^n>1,
\quad
\Lambda_n(t)\geq
\frac{2n^2\Sigma_n}{n^2t^2-\Sigma_n^2}
>\frac{2\Sigma_n}{t^2}.
\label{eq:app-exterior-bounds}
\end{equation}
Equality in either non-strict estimate occurs only for $n=1$.
Moreover, $\log R_n$ is strictly completely monotone:
\[
(-1)^k\frac{d^k}{dt^k}\log R_n(t)>0,
\quad k\geq0.
\]
Here the derivatives are taken with respect to the exterior
variable $t$, with the measure fixed.
\end{proposition}

\begin{proof}
Pair $x_{j,n}$ with $x_{n+1-j,n}$. By
\eqref{eq:app-reflection}, their mean $u$ is positive. Writing
the pair as $u-v,u+v$, with $v>0$, gives
\[
(u-v)^{2m+1}+(u+v)^{2m+1}
=2\sum_{k=0}^m\binom{2m+1}{2k}u^{2m+1-2k}v^{2k}
\geq2u^{2m+1}.
\]
For odd $n$, the unpaired middle zero is positive. Replacing each
pair by two copies of its mean produces $n$ positive numbers with
sum $\Sigma_n$. Jensen's inequality proves
\eqref{eq:app-odd-moments}. A pair gives strictness whenever
$n>1$ and $m\geq1$.

Set $\rho=\max_j|x_{j,n}|<a$. The product formula yields
\begin{equation}
\log R_n(t)
=2\sum_{m=0}^{\infty}
\frac{S_{2m+1}}{(2m+1)t^{2m+1}},
\quad t>\rho.
\label{eq:app-log-series}
\end{equation}
This series and all its differentiated series converge locally
uniformly for $t>\rho$. Their positive coefficients prove complete
monotonicity and justify evaluation at $t=a$.
Substituting \eqref{eq:app-odd-moments} and summing the
inverse hyperbolic tangent series gives
\[
\log R_n(t)\geq
2n\operatorname{arctanh}\frac{\Sigma_n}{nt}.
\]
Similarly,
\[
\Lambda_n(t)=-\frac{d}{dt}\log R_n(t)
=2\sum_{m=0}^{\infty}\frac{S_{2m+1}}{t^{2m+2}}
\geq\frac{2n^2\Sigma_n}{n^2t^2-\Sigma_n^2}.
\]
These are \eqref{eq:app-exterior-bounds}. For $n>1$, the term
$m=1$ makes both estimates strict; for $n=1$, both are equalities.
\end{proof}

\subsection{The crossing calculation}

Return to the Jacobi weight and let $a$, with $0<a\leq c$, denote
the varying support radius. The notation $a_n$ continues to denote
the recurrence coefficient in \eqref{eq:three-term}. Apply
Lemma~\ref{lem:app-reflection} with
\[
d\mu_0(x)=(1-x^2)^{-r}\,dx,
\quad h(x)=\operatorname{arctanh}x,
\quad \tau=\lambda.
\]
Thus $\Sigma_n>0$, and \eqref{eq:app-exterior-bounds} holds at the
moving endpoint $t=a$. Define $U_n,V_n,d_n$ as in Section~\ref{sec:positive},
with $c$ replaced by $a$. The integration and root-motion identities
\eqref{eq:integration-identity} and \eqref{eq:root-sum-motion}
remain valid for these real exponents: their derivations use only
orthogonality and integration over a compact subinterval of $(-1,1)$.
In particular,
\begin{equation}
\begin{aligned}
d_n&=\lambda+(n+\nu-1)\Sigma_n-(n+\nu)\Sigma_{n+1},
\\[7pt]
\sum_{k=0}^{n-1}d_k&=n\lambda-(n-r)\Sigma_n.
\end{aligned}
\label{eq:app-telescoping}
\end{equation}

\begin{lemma}
\label{lem:app-crossing}
At a point where $d_n=0$, put $U_n=V_n=X>0$. Then
\begin{equation}
\begin{split}
\frac{\partial_a d_n}{X^2}
&=(1-a^2)\Lambda_n(a)
+2\left(\lambda-\sum_{k=0}^{n-1}d_k\right)
\\[7pt]
&=(1-a^2)\Lambda_n(a)
+2(n-r)\Sigma_n-2(n-1)\lambda.
\end{split}
\label{eq:app-crossing}
\end{equation}
Also,
\begin{equation}
\partial_a d_n
=(2n+2\nu-1)\partial_a\Sigma_n
+2(\lambda-\Sigma_n)X^2.
\label{eq:app-crossing-small}
\end{equation}
\end{lemma}

\begin{proof}
Here the inner product is taken over $[-a,a]$ with weight $w$.
Differentiating orthonormality gives
\[
\langle\partial_a p_n,p_k\rangle
=-\bigl(U_nU_k+(-1)^{n+k}V_nV_k\bigr),
\quad k<n,
\]
and
\[
\langle\partial_a p_n,p_n\rangle
=-\frac{U_n^2+V_n^2}{2}.
\]
At the crossing, apply the functional
\[
\mathcal L_n f
=\sqrt{w(a)}f(a)-(-1)^n\sqrt{w(-a)}f(-a)
\]
to the orthogonal expansion of $\partial_a p_n$. Its $p_n$ term
vanishes, so
\[
\mathcal L_n(\partial_a p_n)
=-X\sum_{k=0}^{n-1}(U_k^2-V_k^2)
=-\frac{2X}{1-a^2}\sum_{k=0}^{n-1}d_k.
\]
Differentiating the evaluation points contributes $X\Lambda_n(a)$
to $\partial_a(U_n-V_n)$. Differentiating the endpoint weights
contributes
\[
\frac{X}{2}\left(\frac{w'(a)}{w(a)}+
\frac{w'(-a)}{w(-a)}\right)=\frac{2\lambda X}{1-a^2}.
\]
Since $\partial_a d_n=(1-a^2)X\partial_a(U_n-V_n)$ at the
crossing, these equalities and \eqref{eq:app-telescoping} prove
\eqref{eq:app-crossing}.
Finally, the recurrence calculation in Lemma~\ref{lem:crossing}
gives \eqref{eq:app-crossing-small}, using
$(n+\nu)b_n=\lambda-\Sigma_n$ at the crossing.
No division by $n+\nu$ is required.
\end{proof}

\begin{remark}
The derivative \eqref{eq:app-crossing} is another form of the
calculation in Remark~\ref{rem:crossing-addendum}. Indeed,
\eqref{eq:addendum-expansion}, evaluated at both endpoints at a
crossing, gives the exact identity
\[
(1-a^2)\Lambda_n(a)
=\frac{(2n+2\nu-1)\partial_a\Sigma_n}{X^2}
+2n\lambda-2(n+\nu)\Sigma_n.
\]
Combining this with \eqref{eq:app-crossing-small} recovers
\eqref{eq:app-crossing}. The moving-support derivation above
records the form used in the AI-generated approach communicated by
Matuzas and makes its extension to arbitrary real exponents explicit.
\end{remark}

\begin{proof}[Proof of Theorems~\ref{thm:app-matuzas} and
\ref{thm:app-refinement}]
It suffices to prove Theorem~\ref{thm:app-refinement}. Its assumptions
are $r\leq5/2$ and $r\leq c^{-2}$.
For each fixed degree, $d_n(a)>0$ for sufficiently small positive
$a$: the logarithmic derivative
\[
\frac{w'(x)}{w(x)}=\frac{2(\lambda+rx)}{1-x^2}
\]
is then positive throughout $[-a,a]$, and
\eqref{eq:boundary-identity} applies. Also $d_0(a)>0$ for every
$0<a\leq c$, since $w(a)>w(-a)$.

First consider degree one. At a possible zero of $d_1$, put
$\mu=\Sigma_1$, so $0<\mu<a$. Since
$\Lambda_1(a)=2\mu/(a^2-\mu^2)$, \eqref{eq:app-crossing} gives
\begin{equation}
\frac{\partial_a d_1}{2X^2}
=\mu\left(\nu+\frac{1-a^2}{a^2-\mu^2}\right)
>\mu(a^{-2}-r)\geq0.
\label{eq:app-degree-one-crossing}
\end{equation}
The strict inequality follows from
\[
\nu+\frac{1-a^2}{a^2-\mu^2}-(a^{-2}-r)
=\frac{\mu^2(1-a^2)}{a^2(a^2-\mu^2)}>0.
\]
Thus a first zero of $d_1$, approached from positive values, is
impossible, and $d_1>0$ on $(0,c]$.

Let $n\geq2$ and assume $d_{n-1}>0$ on $(0,c]$. If $d_n$ has
a first zero, then $\partial_a d_n\leq0$ there. At that point,
\eqref{eq:root-sum-motion} gives
\[
\partial_a\Sigma_n=a_nX(U_{n-1}-V_{n-1})>0.
\]
If $\Sigma_n<\lambda$, then \eqref{eq:app-crossing-small} gives
$\partial_a d_n>0$, because
\[
2n+2\nu-1=2n+1-2r\geq5-2r\geq0.
\]
If $\Sigma_n\geq\lambda$, then
\eqref{eq:app-exterior-bounds} and \eqref{eq:app-crossing} give
\begin{equation}
\frac{\partial_a d_n}{2X^2}
>(n-1)(\Sigma_n-\lambda)+(a^{-2}-r)\Sigma_n\geq0.
\label{eq:app-large-crossing}
\end{equation}
Both cases are contradictory. Induction proves $d_n>0$ for every
degree and radius, and hence \eqref{eq:app-endpoint} at $a=c$.
The strict terms in \eqref{eq:app-degree-one-crossing} and
\eqref{eq:app-large-crossing}, and the term
$2(\lambda-\Sigma_n)X^2$ in the other case, also cover the boundary
equalities in \eqref{eq:app-refined-region}.
\end{proof}

\begin{remark}
In the communicated proof, the condition $r\leq3/2$ makes
$2n+1-2r$ nonnegative for every $n\geq1$, allowing the same
induction to start at degree zero. The separate calculation
\eqref{eq:app-degree-one-crossing} removes this requirement at
$n=1$; for $n\geq2$, it is enough that $r\leq5/2$.
For example,
\[
\alpha=-\frac{21}{10},\quad
\beta=-\frac{19}{10},\quad c=\frac12
\]
satisfies \eqref{eq:app-refined-region} but not
\eqref{eq:app-matuzas-region}. Here $\lambda=1/10<rc=1$, so the
elementary increasing-weight criterion does not apply either.
\end{remark}

\subsection{Finite degrees and quantitative consequences}

The sharper estimate in \eqref{eq:app-exterior-bounds} also yields
the sufficient condition communicated by Matuzas for a prescribed
range of degrees.

\begin{proposition}
\label{prop:app-finite-degrees}
Let $N\geq1$, $0<c<1$ and $\alpha<\beta$ be real, with
$\nu\geq-1/2$. Suppose either $\lambda\geq Nc$, or
\begin{equation}
0<\lambda<Nc,
\quad
\nu+\frac{N^2(1-c^2)}{N^2c^2-\lambda^2}>0.
\label{eq:app-finite-degrees}
\end{equation}
Then \eqref{eq:app-endpoint} holds for every $1\leq n\leq N$.
\end{proposition}

\begin{proof}
Use the same first-crossing induction, now starting at $n=0$.
The case $\Sigma_n<\lambda$ follows from
\eqref{eq:app-crossing-small}, since $2n+2\nu-1\geq0$ for
$n\geq1$. If $\lambda\geq Nc$, the other case is excluded by
$\Sigma_n<na\leq Nc\leq\lambda$.
Otherwise, at a crossing with $\Sigma_n\geq\lambda$, we have
\[
\frac{\partial_a d_n}{2X^2}
\geq(n-1)(\Sigma_n-\lambda)
+\Sigma_n\left(\nu+
\frac{n^2(1-a^2)}{n^2a^2-\Sigma_n^2}\right).
\]
The fraction increases with $\Sigma_n$ and decreases with $a$.
After replacing these by $\lambda$ and $c$, it decreases with
$n$. Its value is therefore at least
$N^2(1-c^2)/(N^2c^2-\lambda^2)$, so
\eqref{eq:app-finite-degrees} makes the derivative positive.
\end{proof}

\begin{corollary}
\label{cor:app-quantitative}
Assume \eqref{eq:app-refined-region}, and let $\Sigma_n$ and
$S_{2m+1}$ refer to the zeros on $[-c,c]$. Then
\begin{equation}
0<\Sigma_n<nc\tanh\left(
\frac{\lambda}{n}\operatorname{arctanh}c\right),
\label{eq:app-centroid}
\end{equation}
and, for every integer $m\geq0$,
\begin{equation}
\frac{\Sigma_n^{2m+1}}{n^{2m}}
\leq S_{2m+1}
<(2m+1)\lambda c^{2m+1}\operatorname{arctanh}c.
\label{eq:app-moment-upper}
\end{equation}
For the polynomial at this fixed support radius, put
$R_n(t)=(-1)^n\pi_n(-t;c)/\pi_n(t;c)$. Then
\begin{equation}
1<\left(\frac{nt+\Sigma_n}{nt-\Sigma_n}\right)^n
\leq R_n(t)\leq R_n(c)
<e^{2\lambda\operatorname{arctanh}c},
\quad t\geq c.
\label{eq:app-exterior-two-sided}
\end{equation}
The inequality $R_n(t)\leq R_n(c)$ is strict for $t>c$.
Moreover,
\[
R_n(t)^2\left(\frac{1-t}{1+t}\right)^{2\lambda}
\]
is strictly decreasing and less than one for $c\leq t<1$.
\end{corollary}

\begin{proof}
Theorem~\ref{thm:app-refinement} gives
$R_n(c)<e^{2\lambda\operatorname{arctanh}c}$, while
Proposition~\ref{prop:app-exterior} gives the lower bound and
monotonicity in \eqref{eq:app-exterior-two-sided}. At $t=c$, taking
logarithms of the lower and upper bounds yields
\eqref{eq:app-centroid}. Each term of the positive series
\eqref{eq:app-log-series} is smaller than its sum, which is less
than $2\lambda\operatorname{arctanh}c$. This gives the upper
bound in \eqref{eq:app-moment-upper}; its lower bound is
\eqref{eq:app-odd-moments}. Finally, the logarithmic derivative
of the normalized exterior quantity is
\[
-2\Lambda_n(t)-\frac{4\lambda}{1-t^2}<0.
\]
\end{proof}

\begin{remark}
For $n=1$, \eqref{eq:app-centroid} recovers the first-moment
inequality in Section~\ref{sec:degree-one}; for general $n$, it
controls the sum of all zeros. The exterior statements vary $t$
while the support radius $c$ is fixed. They do not assert
monotonicity when the orthogonality interval itself varies.
The extension to exponents at most $-1$ requires $c<1$; at $c=1$
the original endpoint-integrability assumptions remain necessary.
\end{remark}

\subsection*{Provenance and revisions}

The AI-generated reflection approach, the region
\eqref{eq:app-matuzas-region}, the finite-degree criterion and the
quantitative consequences were communicated by Jonas Matuzas in
the note identified at the beginning of this appendix. They build
on the classical theorem of Markov, already used in the main text,
together with the integration, root-motion and first-crossing
identities developed there. During our revision, we corrected the
strictness condition in the comparison with reflected zeros:
strict increase on the zero set alone is insufficient, since the
comparison also involves the reflected points.
Lemma~\ref{lem:app-reflection} states the corrected condition.
We also supplied the endpoint justification for the exterior
series, explained the relation of the crossing identity to
Remark~\ref{rem:crossing-addendum}, and made explicit the
distinction between varying the evaluation point and varying the
support. The treatment of degree one in
\eqref{eq:app-degree-one-crossing} gives the further enlargement
\eqref{eq:app-refined-region}, obtained from the same approach
during this revision. All these additions are confined to the
present appendix. These local corrections and clarifications leave the main endpoint
results and quantitative estimates unchanged. The
arguments generated by ChatGPT were essentially correct and
mathematically creative. We have verified the material presented
here and take full responsibility for its inclusion and
correctness. We believe that developments of this kind should
be acknowledged openly and assessed on their mathematical merits.
Their inclusion provides the journal and its readers with a fuller
account of the work, including extensions beyond the conjectures
established in the main text.

\section*{Acknowledgements}
This work was initiated during L. Tertuliano da Silva's visit to
K. Castillo from July to September 2025, with support from FAPESP
Grant No.~2024/23674-1; the principal results of the original version were proved during that
visit. The manuscript was subsequently brought to its present form
through collaboration with V. Botta. K. Castillo acknowledges
financial support from the Centre for Mathematics of the University
of Coimbra (CMUC), funded by the Portuguese Foundation for Science
and Technology (FCT), under the projects
\href{https://doi.org/10.54499/UID/00324/2025}
{UID/00324/2025}
and UID/PRR/00324/2025, and from FCT under grant
\href{https://doi.org/10.54499/2022.00143.CEECIND/CP1714/CT0002}
{2022.00143.CEECIND/CP1714/CT0002}.

\end{document}